%% file: sl_2_paper_v8.tex
\documentclass[10pt, letter]{article}

\usepackage{amsfonts}
\usepackage{amsmath}
\usepackage{amssymb}
\usepackage{amsthm}
\usepackage{stmaryrd}
\usepackage[alphabetic]{amsrefs}
\usepackage{calc}
\usepackage{graphicx}
\usepackage{hyperref}
\usepackage[left=1in,top=1in,right=1in,foot=1in]{geometry}
\usepackage{todonotes}
\usepackage{microtype}
\usepackage{tikz-cd}

\input{preamble.tex}

\newcommand\Fl{{\mathcal F} l}

\theoremstyle{remark}
\newtheorem{rem}{Remark}

\theoremstyle{theorem}
\newtheorem{conjecture}{Conjecture}

\theoremstyle{definition}

\providecommand{\keywords}[1]{\textbf{\textit{Keywords---}} #1}

\title{On Lusztig's asymptotic Hecke algebra for $\SL_2$}
\date{\today}
\author{Stefan Dawydiak \thanks{Department of Mathematics, University of Toronto, Toronto, ON M5S 2E4
Canada; email \texttt{stefand@math.utoronto.ca}}}

\begin{document}
\maketitle
\begin{abstract}
Let $G$ be a split connected reductive algebraic group, let $H$ be the corresponding affine Hecke algebra, and let $J$ be the corresponding asymptotic Hecke algebra in the sense of Lusztig. When $G=\SL_2$, and the parameter $q$ is specialized to a prime power,
Braverman and Kazhdan showed recently that for generic values of $q$, $H$ has
codimension two as a subalgebra of $J$, and described a basis for
the quotient in spectral terms. In this note we write these functions
explicitly in terms of the basis $\{t_w\}$ of $J$, and further invert the
canonical isomorphism between the completions of $H$ and $J$, obtaining
explicit formulas for each basis element $t_w$ in terms of the basis
$\{T_w\}$ of $H$. We conjecture some properties of this expansion for
more general groups. We conclude by using our formulas to prove that $J$
acts on the Schwartz space of the basic affine space of $\SL_2$, and produce some
formulas for this action.
\end{abstract}
\section{Introduction}
\subsection{The asymptotic Hecke algebra}
\label{section the asymptotic Hecke algebra}For $G$ a connected reductive algebraic group, a specialization of the affine Hecke algebra $H$ corresponding to the affine Weyl group $\tilde{W}$ of $G$ plays an important role in the representation theory of $G(F)$ for a $p$-adic field $F$. Explicitly, given a smooth representation $\pi$ of $G(F)$, a function $f\in H$ yields an endomorphism $\pi(f)$ of $\pi^I$, where $I$ is the Iwahori subgroup of $G$.

In \cite{affineII}, Lusztig defined the asymptotic Hecke
algebra $J$, which is a $\Z$-algebra with basis $\{t_z\}_{z\in\tilde{W}}$ equipped with an injection
$\phi\colon H\into J\otimes_\Z\mathcal{A}$ given by
\[
\phi\left(\sum_{x\in\tilde{W}}b_x C_x\right)=\sum_{\substack{x,z\in\tilde{W} \\ d\in\mathcal{D},~a(d)=a(z)}}b_xh_{x,d,z}t_z,
\]
where $\mathcal{D}$ is the set of distinguished involutions and $a$ is 
Lusztig's $a$-function; see \S \ref{section preliminaries} and Definition \ref{a function definition}. 
Multiplication (see Remark \ref{J ring structure remark}) in $J$, and the
definition of the map $\phi$ is given combinatorially in terms
of the structure constants for $H$ written in the $\{C_w\}$ basis. It was also shown in \cite{affineII} that $\phi$ is
an isomorphism after a certain completion, whose details we
recall in \S\ref{the map phi section}.

In \cite{BK}, the authors found an interpretation of $J$ as
certain $I\times I$-invariant functions on $G(F)$ and described the 
corresponding endomorphisms $\pi(f)$.

The purpose of this paper is to study the map $\phi$ in more detail (in the case of $\SL_2$) in order to obtain an explicit, as opposed to spectral, description of the elements of $J$ as functions on $G(F)$.
In what follows it will be convenient to twist $\phi$ by an involution $j$ of $H$ described in \S\ref{section preliminaries}. 
Then our first main result is as follows: we give a formula for $(\phi\circ j)^{-1}(t_w)$ for all $w$ by an explicit calculation in a self-contained way. The resulting formulas are given in Theorem \ref{phi inverse formula theorem} and Corollary \ref{tw Tw expansion corollary}.
As a byproduct we obtain the following result:
\begin{theorem}
\label{expansion}
\begin{enumerate}
\item
For any $w$ the element $(\phi\circ j)^{-1}(t_w)\in \mathcal{H}$ has the form
\[
\sum a_{w,x} C'_x
\]
where $a_{w,x}$ is a polynomial in $q^{-\frac{1}{2}}$. Moreover, $(-1)^{\ell(x)}a_{w,x}$ has nonpositive
integer coefficients.
\item
For any $w$ the element $(\phi\circ j)^{-1}(t_w)\in\mathcal{H}$ has the form
\[
\sum b_{w,x} T_x
\]
where $(q+1)b_{w,x}$ is a polynomial in $q^{-\frac{1}{2}}$. 
\end{enumerate}
\end{theorem}
Let us remark that if we work with a finite Coxeter group instead of an affine one, then while the second assertion 
of Theorem \ref{expansion} remains true (in general $q+1$ must be replaced by the Poincar\'e polynomial of the 
corresponding flag variety), the first assertion is wrong in that case. In fact, it is clear that for finite 
Coxeter groups if some of the coefficients $b_{w,x}$ are genuine rational functions (i.e. not polynomials) then 
the same will also be true for some of the $a_{w,x}$.

We conjecture that similar statements hold more generally.
\begin{conjecture}
\label{function conjecture}
For any split connected reductive group $G$ and any $w\in \tilde{W}$,
we have 
\[
(\phi\circ j)^{-1}(t_w)=\sum a_{w,x} C'_x
\]
where $a_{w,x}$ is a polynomial in $q^{-\frac{1}{2}}$ such that $(-1)^{\ell(x)}a_{w,x}$ has nonpositive
coefficients.
Similarly, we conjecture that
\[
(\phi\circ j)^{-1}(t_w)=\sum b_{w,x} T_x
\]
where $(\sum_{w\in W} q^{\ell(w)})b_{w,x}$ is a polynomial in $q^{-1/2}$ (note that the sum in parentheses  is over the finite Weyl group).
\end{conjecture}
Conjecture \ref{function conjecture} (if true) is very interesting from a geometric point of view, and one can hope that the coefficients carry representation-theoretic information. More specifically, it would be 
extremely interesting to categorify $J$ with its basis $\{t_w\}$. By this we mean the following. Let ${\mathcal K}={\mathbb C}((z)), {\mathcal O}={\mathbb C}\llbracket z\rrbracket$. Consider the ind group-scheme $G(\mathcal K)$. Let $\Fl=G(\mathcal K)/I$ denote the affine flag variety. Then the Iwahori-Hecke 
algebra $H$ is the Grothendieck ring of the bounded derived category of mixed $I$-equivariant constructible 
sheaves on $\Fl$. Under this isomorphism the elements $C_x'$ correspond to the classes of irreducible perverse 
sheaves. The above conjecture suggests that the elements $t_w$ correspond to some canonical ind-objects in the 
above derived category. Moreover, these objects should have the property that every simple perverse 
sheaf appears there, shifted according to Lusztig's $a$ function (see Definition \ref{a function definition}). It would be extremely interesting to find a construction of these 
objects. 

The key simplification in type $\tilde{A}_1$ that allows the 
computations carried out in this note is the simple nature of the 
affine Weyl group and that the Kazhdan-Lusztig 
polynomials are all constant and equal to one, so that each $C'_w$
is a constant function. Geometrically, this corresponds to smoothness 
of $I$-orbit closures in $\Fl$. Exact 
formulas for the elements $t_w$ seem to be unlikely in higher rank, when these 
simplifications are not present.
\subsection{Further results}
In \S\ref{tw section} we show in an
elementary way that $J$ acts on $C_c^\infty (G/N)^I$, reproving in an elementary (in that we make make no serious use of the theory of harmonic analysis on $p$-adic groups, and use no algebraic geometry whatsoever) way
a result of \cite{BK}, and that $J$ lies
in the Harish-Chandra Schwartz space of $G$. These results are
recorded as Propositions \ref{t1 plane action trivial prop} and
\ref{ts0 plane action formula prop}, and Theorem \ref{J action 
theorem}. Let $\mathcal{S}_c=C^\infty_c(G/N)$ and let $\mathcal{S}$
be the Schwartz space of the basic affine space as in \cite{BKBasicAffine}.
In \cite{BK}, it is proved that the direct summand $J_0$ of $J$ 
corresponding to the big cell in $\tilde{W}$ is exactly the space 
of endomorphisms of $\mathcal{S}^I$ commuting with all Fourier transforms and all translations by cocharacters of a fixed maximal torus in $G$, and that $J_0\cdot\mathcal{S}_c^I=\mathcal{S}^I$.
In this way knowledge of $\mathcal{S}^I$ is equivalent to knowledge of $J_0$,
which in the case of $\SL_2$ is just $J_0=\spn{\{t_w\}}_{w\neq 1}$.
\subsection{Acknowledgements}
The author thanks Alexander Braverman for many helpful conversations and for introducing him to this material, 
and the Center for Advanced Studies at the Skolkovo Institute of Science and Technology for their 
hospitality during the period when this work was done. The author also thanks Kostya Tolmachov for helpful discussions.
\section{Formulas for the map $\phi$}
\subsection{Preliminaries}
\label{section preliminaries}
Throughout, $\pi$ is a uniformizer of a fixed non-archimedean local field
$F$ with ring of integers $\Oo$, and $q$ is the cardinality of the
residue field $\Oo/\pi\Oo$ (although until \S\ref{tw section} we can also view it as an indeterminate). We shall write $G=\SL_2$ as algebraic groups.
When there is no room for confusion, we write $G$ for $G(F)$ as well.
We fix the Borel subgroup $B$ of upper triangular matrices, and write $I\subset G(\Oo)$
for the corresponding Iwahori subgroup. Put $\tilde{W}$ for the affine Weyl group of $G$ ,with length function $\ell$ and set $S$ of simple reflections. Let $H$ be the Iwahori-Hecke algebra of $G$, over the ring $\mathcal{A}=\Z[q^{\frac{1}{2}},q^{-\frac{1}{2}}]$.  We recall that $H$ has a  basis $\{T_w\}_{w\in\tilde{W}}$, where multiplication is defined by
relations $T_wT_{w'}=T_{ww'}$ if $\ell(ww')=\ell(w)+\ell(w')$ and quadratic relation $(T_s+1)(T_s-q)=0$ for $s\in S$.
Additionally, we have the Kazhdan-Lusztig basis
\[
C_w=\sum_{y\leq w}(-1)^{\ell(w)-\ell(y)}q^{\frac{\ell(W)-\ell(y)}
{2}}P_{y,w}(q^{-1})q^{-\frac{\ell(y)}{2}} T_y
\]
and the basis $\{C'_w\}_{w\in\tilde{W}}$, which we recall is related to the $\{C_w\}_{w\in\tilde{W}}$ basis by $C'_w=(-1)^{\ell(w)}j(C_w)$. Here $j$ is the algebra
involution on $H$ defined in \cite{KL79} by $j(\sum
a_wT_w)=\sum \bar{a_w}(-1)^{\ell(w)}q^{-\ell(w)}T_w$, where $\bar{(\,)}\colon\mathcal{A}\to\mathcal{A}$ is the involution defined by $\overline{q^{\frac{1}{2}}}=q^{-\frac{1}{2}}$. The bar involution of $\mathcal{A}$ extends to the bar involution of $H$, and we have $\bar{C_w}=C_w$
and $\bar{C'_w}=C'_w$ for all $w$.
Several
definitions will be given in terms of the structure constants of $
H$ in the basis $\{C_w\}$, and we write $h_{x,y,z}$ to
mean those elements of $\mathcal{A}$ such that $C_xC_y=\sum_{z}
h_{x,y,z}C_z$.

Let $\alpha\colon\diag(a,a^{-1})\mapsto a^2$ be the positive root of $\SL_2$, and $\alpha^\vee$ the corresponding coroot. Write $X_*(A)$ for the cocharacter group of the maximal torus $A$ of diagonal matrices.
From now on, $\tilde{W}=W\ltimes X_*(A)=W\ltimes\Z\langle\alpha^\vee
\rangle$ is the affine Weyl group for $G=\SL_2$, with
fixed presentation $\tilde{W}=\genrel{s_0,s_1}{s_0^2=s_1^2=1}$.
We write $S=\{s_0,s_1\}$, with $s_1$ the
affine reflection, so that $W=\langle s_0\rangle$ is the finite Weyl
group. When working with this presentation, all the words we write down will be reduced.
The identification between this presentation and the semidirect
product realization of $\tilde{W}$ sends $s_0$ to the simple
reflection $s_\alpha$ corresponding to $\alpha$, and $s_1$ corresponds to $s_\alpha\pi$, where $
\pi=\pi^{\alpha^\vee}$. Our convention is that $\alpha$ is dominant, so that dominant coweights correspond to
positive integers, with $\pi^n=\pi^{n\alpha^\vee}=(s_0s_1)^n$ being
dominant,and $\pi^{-n}=(s_1s_0)^n$ being antidominant.
The distinguished involutions in $\tilde{W}$ are $
\mathcal{D}=\{1,s_0,s_1\}$. We remark that as an abstract group, $\tilde{W}$ is the infinite dihedral group, with $s_0$ and $s_1$
playing symmetric roles. However, as seen above, under the identification we have fixed, the finite and affine simple reflections play different roles. There is however an automorphism of $H$ 
exchanging $T_{s_0}$ and $T_{s_1}$, see \S\ref{subsection convolutions}.
In our special case, we have 
\[
C'_w=q^{-\frac{\ell(w)}{2}}\sum_{y\leq w}T_y,
\]
where $\leq$ is the strong Bruhat order \textit{i.e.} $y\leq w$ if and only if
after writing a reduced word for $w$ and deleting
some letters, we obtain a word for $y$.
\begin{ex}
We have $C'_e=1=T_e$ is the unit in $H$, where $e$ is the unit element in $\tilde{W}$, and
\[
C_{s_0s_1s_0}'=q^{-\frac{3}{2}}\left(T_{s_0s_1s_0}+T_{s_1s_0}
+T_{s_0s_1}+T_{s_0}+T_{s_1}+1\right).
\]
\end{ex}
\subsection{The map $\phi$}
\label{the map phi section}
\begin{prop}[\cite{affineII}, \S 2.4]
\label{phi definition}
The map $\phi\colon H\to J\otimes_\Z\mathcal{A}$ defined in \S \ref{section the asymptotic Hecke algebra} is a morphism of algebras.
\end{prop}
We now recall the details of the completion mentioned above. Let $\hat{\mathcal{A}}$ be the ring of formal Laurent
series in $q^{\frac{1}{2}}$, and let $\hat{\mathcal{A}}^+$ be the
ring of formal power series in $q^{\frac{1}{2}}$. We obtain a completion  $\mathcal{H}$ of $H$ whose elements are (possibly infinite) $\hat{\mathcal{A}}$-linear
combinations $\sum_{x}b_xC_x$ such that $b_x\to 0$ in the $(q)$-adic
topology on $\hat{\mathcal{A}}^+$ \textit{i.e.} such that for any $N>0$, $b_x\in (q^{\frac{1}{2}})^N\hat{\mathcal{A}}^+$ for $\ell(x)$ sufficiently large. When working with the basis $\{C'_w\}_{w\in\tilde{W}}$, we complete with respect to the negative powers of $q$. The involution $j$ naturally extends to a homeomorphism
between these different completions. In the same way, we obtain a completion $\mathcal{J}$ of $J\otimes_\Z\mathcal{A}$.
The definition of $\phi$ (see Proposition \ref{phi definition}) carries over verbatim, yielding an isomorphism $\phi\colon\mathcal{H}\overset{\sim}{\to}\mathcal{J}$. 

Over the course of the next three lemmas, we shall see that the definition of this
map simplifies considerably in our case. We first recall two special 
cases of results of Lusztig. We refer to the exposition in 
\cite{LusUnequal} for this material. There Lusztig writes $T_w$ for our $q^{-\frac{\ell(w)}{2}}T_w$, $c_w$ for our $C'_w$, and in our case $p_{y,w}=q^{\frac{-\ell(w)+\ell(y)}{2}}$.
We write $\mathcal{R}(w)=\sets{s\in S}{ws<w}$. If $w=rs_i$ is nontrivial, $\mathcal{R}(w)=\{s_i\}$ is a singleton.
\begin{lem}[\cite{LusUnequal}, Corollary 6.7]
\label{joined lemma 4.3}
Let $w\in\tilde{W}$ and $s=s_i$. Then
\[
C_wC_s=
\begin{cases}
-\left(q^{\frac{1}{2}}+q^{-\frac{1}{2}}\right)C_w&\text{if}~s\in\mathcal{R}(w)\\
\sum_{\substack{|\ell(w)-\ell(y)|=1 \\ ys<y}}C_y
&\textit{if}~s\not\in\mathcal{R}(w)
\end{cases}.
\]
\end{lem}
\begin{dfn}[Lusztig's $a$ function.]
\label{a function definition}
For $w\in\tilde{W}$, define $a(w)$ to be the smallest
integer such that $(-q)^{\frac{a(w)}{2}}h_{x,y,w}\in\mathcal{A}^+$ for all $x,y\in\tilde{W}$.
\end{dfn}
\begin{lem}[\cite{LusUnequal}, \S 13.4, Lemma 13.5, Proposition 13.7]
\label{a function lemma}
Let $w\in\tilde{W}$. If $w=1$, then $a(w)=0$. Otherwise
$a(w)=1$.
\end{lem}
Assembling Lemmas \ref{joined lemma 4.3} and \ref{a function lemma}, we can describe $\phi$
explicitly.
\begin{lem}
\label{phi lemma}
Let $i\neq j$ and $i,j\in\{0,1\}$. Then
\[
\phi(C_{s_i})=-\left(q^{\frac{1}{2}}+q^{-\frac{1}{2}}\right)t_{s_i}+t_{s_is_j}.
\]
More generally, if $\ell(w)\geq 2$ and $w=rs_i$, then
\[
\phi(C_w)=-\left(q^{\frac{1}{2}}+q^{-\frac{1}{2}}
\right)t_{rs_i}+t_r+t_{rs_is_j}.
\]
\end{lem}
\begin{proof}
We need only note that the condition $ys_j<y$ from Lemma
\ref{joined lemma 4.3} implies $y$ ends in $s_j$.
\end{proof}

Recall that the unit in $J$ is $1_J=t_{s_0}+t_{s_1}+t_1$, the sum of the basis elements corresponding to distinguished involutions. As $\phi$ preserves units, we have
$\phi(C_1)=t_1+t_{s_1}+t_{s_0}$.
\begin{dfn}
If $w$ and $y$ are elements in  $\tilde{W}$, we say that $w$
\emph{starts with} $y$ if we have reduced expressions $y=s_{i_1}\cdots s_{i_n}$ and
$w=s_{i_1}\cdots s_{i_n}s_{i_{n+1}}\cdots s_{i_{n+m}}$ for some $m
\geq 0$.
\end{dfn}
\begin{lem}
\label{phi of si lemma}
We have
\[
\phi\left(\sum_{\substack{w\in\tilde{W} \\ w~\text{\emph{starts with}}~s_0}}q^{\frac{\ell(w)}{2}}C_w\right)= -t_{s_0},
\]
and likewise
\[
\phi\left(\sum_{\substack{w\in\tilde{W} \\ w~\text{\emph{starts with}}~s_1}}q^{\frac{\ell(w)}{2}}C_w\right)= -t_{s_1}.
\]
\end{lem}
\begin{proof}

Under $\phi$, the infinite sum $\sum_{\substack{w\in\tilde{W} \\ w~\text{starts with}~s_0}}q^{\frac{\ell(w)}{2}}C_w$ is sent to
\begin{align}
&q^{\frac{1}{2}}\left(-\left(q^{\frac{1}{2}}+q^{-
\frac{1}{2}}\right)t_{s_0}+t_{s_0s_1}\right)
\label{first line s0 sum}
\\
&+
q\left(-\left(q^{\frac{1}{2}}+q^{-
\frac{1}{2}}\right)t_{s_0s_1}+t_{s_0}+t_{s_0s_1s_0}\right)
\\
&+
q^{\frac{3}{2}}\left(
-\left(q^{\frac{1}{2}}+q^{-
\frac{1}{2}}\right)t_{s_0s_1s_0}+t_{s_0s_1}+t_{s_0s_1s_0s_1}
\right)\label{last line s0 sum}
\\
&+\cdots
\nonumber.
\end{align}
By Lemma \ref{phi lemma} again, cancellation of terms
appearing in $\phi(C_w)$ with $\ell(w)=n$ can occur only against
terms appearing in $\phi(C_m)$ with
$|n-m|=1$, and we see that after cancellations between the terms
on lines \eqref{first line s0 sum} through \eqref{last line
s0 sum}, corresponding to lengths at most $3$, the sum stands as
\[
-t_{s_0}-q^2t_{s_0s_1s_0}+t_{s_0s_1s_0s_0s_1}+\text{terms from
longer words}.
\]
Further, if $r$ starts with $s_0$ and $w=rs_0$, the term $-q^{\frac{\ell(w)-1}{2}}q^\frac{1}{2}t_r$ from $
\phi(C_r)$ cancels with the term $q^{\frac{\ell(w)}{2}}t_r$ coming
from $\phi(C_w)$, and the term $q^{\frac{\ell(w)-1}{2}}t_w$ from $\phi(C_r)$ cancels with the term
$-q^{\frac{\ell(w)}{2}}q^{-\frac{1}{2}}t_w$ in $\phi(C_w)$. Likewise the terms $-q^{\frac{\ell(w)}
{2}}q^{\frac{1}{2}}t_w$ cancels with a term from $\phi(C_{ws_1})$ and $q^{\frac{\ell(w)}{2}}t_{ws_1}$
cancels with the term $-q^{\frac{\ell(w)+1}{2}}q^{-\frac{1}{2}}t_{ws_1}$ from $\phi(C_{ws_1})$. The
case for $w$ ending in $s_1$ is identical, and cancellations happen between terms from two words
ending both in $s_0$. The calculation for $t_{s_1}$ is identical.
\end{proof}
The formula for $\phi^{-1}$ is implicit in the proof Lemma \ref{phi
of si lemma}. Indeed, the lemma upgrades to
\begin{lem}
\label{phi inverse pair lemma}
Let $y=s_{i_1}\cdots s_{i_n}$, and let $i=i_n$. Then
\[
\phi\left(\sum_{\substack{w\in\tilde{W} \\ w~\text{\emph{starts with}}~y}}
q^{\frac{\ell(w)}{2}}C_w\right)= -q^{\frac{\ell(y)-1}{2}}t_y+
q^{\frac{\ell(y)}{2}}t_{ys_i}.
\]
\end{lem}
\begin{proof}
Direct calculation as in Lemma \ref{phi of si lemma}. Let $s_j$ be the
generator that is not $s_i$. Then the first terms are
\[
q^{\frac{\ell(y)}{2}}\left(-\left(q^{\frac{1}{2}}+q^{-\frac{1}{2}}
\right)t_{y}+t_ys_i+t_{ys_j}\right)
+
q^{\frac{\ell(y)+1}{2}}\left(-\left(q^{\frac{1}{2}}+q^{-\frac{1}{2}}
\right)t_{ys_j}+t_{y}+t_{ys_js_i}\right)+\cdots,
\]
and the cancellations in the proof of Lemma
\ref{phi of si lemma} pick up from this point, leaving only
$-q^{\frac{\ell(y)-1}{2}}t_y+q^{\frac{\ell(y)}{2}}t_{ys_i}$.
\end{proof}
We can therefore calculate $\phi^{-1}(t_{y})$ up to an error
term of length $\ell(ys_i)<\ell(y)$. Given that we can calculate
$\phi(t_{s_i})$, we can cancel the error terms inductively, yielding
a formula for $\phi^{-1}$.
\begin{theorem}
\label{phi inverse formula theorem}
Let $y=s_{i_1}s_{i_2}\cdots s_{i_n}$ so that $\ell(y)=n>0$,
and for $k\leq n$, write $y_k=s_{i_1}\cdots s_{i_k}$. Then
\[
-q^{\frac{n-1}{2}}\phi^{-1}(t_y)=\sum_{k=1}^{n}q^{n-k}
\sum_{\substack{w\in\tilde{W} \\ w~\text{\emph{starts with}}~y_k}}q^{\frac{\ell(w)}{2}}C_w
\]
%
\end{theorem}
\begin{proof}
It suffices to prove that the images of the left-hand side and of the right-hand side under $\phi$ are equal. To do this, apply Lemma \ref{phi inverse pair lemma} to the last $\ell(y)-1$
summands and Lemma \ref{phi of si lemma} to the first.
\end{proof}
\begin{ex}
We calculate $\phi^{-1}(t_{s_0s_1s_0s_1})$, where $n=4$.
Under $\phi$,
\begin{align*}
&
q^2C_{s_0s_1s_0s_1}+q^{\frac{5}{2}}C_{s_0s_1s_0s_1s_0}+q^3C_{s_0s_1s_0s_1s_0s_1}+\cdots
\\
&+
q\left(
q^{\frac{3}{2}}C_{s_0s_1s_0}+
q^2C_{s_0s_1s_0s_1}+q^{\frac{5}{2}}C_{s_0s_1s_0s_1s_0}
+q^3C_{s_0s_1s_0s_1s_0s_1}+\cdots\right)
\\
&+
q^2\left(
q^C_{s_0s_1}+
q^{\frac{3}{2}}C_{s_0s_1s_0}+
q^2C_{s_0s_1s_0s_1}+q^{\frac{5}{2}}+C_{s_0s_1s_0s_1s_0}
+q^3C_{s_0s_1s_0s_1s_0s_1}+\cdots
\right)\\
&+
q^{3}
\sum_{\substack{w\in\tilde{W} \\ w~\text{starts with}~s_0}}q^{\frac{\ell(w)}{2}}C_w
\end{align*}
is sent to
\[
q^2t_{s_0s_1s_0}-q^{\frac{3}{2}}t_{s_0s_1s_0s_1}+q^{\frac{5}{2}}
t_{s_0s_1}-q^2t_{s_0s_1s_0}
+q^3t_{s_0}-q^\frac{5}{2}t_{s_0s_1}-q^3t_{s_0}=-q^{\frac{3}{2}}
t_{s_0s_1s_0s_1}.
\]
\end{ex}
\begin{cor}
\label{tw Tw expansion corollary}
If $y$ is as above, we have
\begin{multline*}
-q^{\frac{1-n}{2}}(\phi\circ j)^{-1}(t_y)=
\sum_{k=1}^{n}q^{k-n}\left(
\sum_{\substack{w\in\tilde{W} \\ w~\text{\emph{starts with}}~y_k}}
\frac{(-1)^{\ell(w)}q^{-\ell(w)+1}}{1+q}T_w+
\sum_{\substack{w\in\tilde{W} \\ w~\text{\emph{does not start with}}
~y_k \\ \ell(w)\geq k}}
\frac{(-1)^{\ell(w)+1}q^{-\ell(w)}}{1+q}T_w
\right.
\\
+
\left.
\frac{(-1)^{k}q^{-k+1}}{1+q}
\sum_{\substack{w\in\tilde{W} \\ w~\text{\emph{does not start with}}
~y_k \\ \ell(w)< k}}
T_w
\right)
.
\end{multline*}
\end{cor}
The constant factor $q(1+q)^{-1}$ in each summand appears as $\sum_{n=0}
^{\infty}(-1)^nq^{-n}$. 

\subsection{The functions $f$ and $g$}
\label{section the functions f and g}
In \cite{BK}, Braverman and Kazhdan gave a spectral definition of two functions $f$ and $g$ on $G$, which viewed as elements in $J$ which span $J/H$ when $q$ is specialized to a prime power. 

They are
\[
f=
T_1+T_{s_0}+\sum_{n=1}^{\infty}q^{-2n}\left(
T_{(s_1s_0)^n}+T_{s_0(s_1s_0)^n}-q\left(
T_{(s_0s_1)^n}+T_{s_1(s_0s_1)^n}
\right)
\right)
\]
and
\[
g=\sum_{w\in\tilde{W}}(-1)^{\ell(w)}q^{-\ell(w)}T_w
\]
We find their images under $\phi$ and show they lie in $J$ by explicit calculation in Theorem \ref{images theorem}. 

By \cite{BK} equation 4.1, we have $J=\End(\mathrm{St}^I)\oplus J_0$, where $\mathrm{St}$ is 
the Steinberg representation of $\SL_2$, and $J_0$ is the algebra of endomorphisms of $C_c^\infty(F^2)^I$ that commute with translation and Fourier transform, see  \S\ref{Schwartz space subsection}. The function $g$ is the matrix coefficient of $\mathrm{St}^I$ and induces an integral operator spanning $\End(\mathrm{St}^I)$. The function $f$
does not have such a nice description, but the closely-related function $\tilde{f}$ (see equation \eqref{tilde f prepared form})
is defined to be constant on $I$-orbits on $G(\Oo)\rquotient G(F)$
by putting $\tilde{f}\restriction_X=(-q)^{-\dim X-1}$ for $I$-orbits $X$. We conjecture that $\tilde{f}$ thus defined lies in $J$ for any connected reductive group $G$.
\begin{rem}
The function $f$ is defined in \cite{BK} directly as a function on $
\SL_2(\Oo)\rquotient \SL_2(F)/I$. Our definition is equivalent, as can be seen by writing
\[
\SL_2(\Oo)\cdot\diag(t^n,t^{-n})\cdot I=I\cdot
\diag(\pi^n,\pi^{-n})\cdot I\coprod I\cdot
\begin{pmatrix}
0 & -1 \\
1 & 0
\end{pmatrix}
\diag(\pi^{-n},\pi^n)
\cdot I.
\]
\end{rem}
It is easy to rewrite elements given in the $T_w$ basis to elements given in the $C'_w$ basis; the
change of basis is ``upper-triangular with monomial entries." Precisely, we have the following
\begin{prop}
\label{change of basis prop}
We have
\[
T_w=\sum_{y\leq w}q^{\frac{\ell(y)}{2}}(-1)^{\ell(w)-\ell(y)}C_y'.
\]
\end{prop}
\begin{proof}
Clearly the proposition is true for $\ell(w)=0$, and for $
\ell(w)=1$. Now write $w=s_i rs_j$, so that
\[
C'_w=q^{-\frac{\ell(w)}{2}}\left(T_w+T_{rs_j}+T_{s_ir}+\cdots\right)=q^{-\frac{\ell(w)}{2}}\left(T_w+T_{rs_j}+q^{\frac{\ell(s_i r)}{2}}C'_{s_ir}\right)
\]
whence
\[
q^{\frac{\ell(w)}{2}}C'_w-q^{\frac{\ell(s_i r)}{2}}C'_{s_ir}=T_w+T_{rs_j}.
\]
The claim follows by induction on $\ell(w)$.
\end{proof}
We can now rewrite the functions $f$ and $g$ in the $C'_w$ basis, in preparation for applying $\phi\circ j$ to them. In the case of $g$, we have
\[
g=\sum_{w\in\tilde{W}}(-1)^{\ell(w)}q^{-\ell(w)}T_w=\sum_{w\in\tilde{W}}(-1)^{\ell(w)}q^{-\ell(w)}
\left(\sum_{y\leq w}q^{\frac{\ell(y)}{2}}(-1)^{\ell(w)-\ell(y)}C'_y\right),
\]
and we see that the coefficient $b_w$ of $C'_w$ is a power series in $q^{-\frac{1}{2}}$ of order
$q^{\frac{\ell(w)}{2}}$. Indeed, $C'_w$ will appear once in the expansion of $T_w$, and then twice
for each length greater than $\ell(w)$, and thus
\[
b_w=(-1)^{\ell(w)}q^{-\ell(w)}q^{\frac{\ell(w)}{2}}+
2\left(\sum_{n=\ell(w)+1}^{\infty}(-1)^n(-1)^{n-\ell(w)}q^{\frac{\ell(w)}{2}}q^{-n}\right).
\]
For $z\in\tilde{W}$ such that $\ell(z)=n\geq\ell(w)$, $(-1)^nq^{-n}$ is the coefficient of $T_z$ in
rewriting $g$, and $(-1)^{n-\ell(w)}q^{\frac{\ell(w)}{2}}$ is the coefficient of $C_w$ in the
expansion of $T_z$ according to Proposition \ref{change of basis prop}. Therefore
\[
b_w=(-1)^{\ell(w)}q^{-\frac{\ell(w)}{2}}\left(1+2\frac{q^{-1}}{1-q^{-1}}\right),
\]
and so
\begin{equation}
\label{g expression C'}
g=\left(1+2\frac{q^{-1}}{1-q^{-1}}\right)\sum_{w\in\tilde{W}}(-1)^{\ell(w)}q^{-\frac{\ell(w)}{2}}
C'_w.
\end{equation}
We note that $1+2\frac{q^{-1}}{1-q^{-1}}=1+2q^{-1}+2q^{-2}+\cdots=\sum_{w\in\tilde{W}}q^{-\ell(w)}$ is a unit in $\Z\llbracket q^{-
\frac{1}{2}}\rrbracket$.

Rewriting the function $f$ is simpler, in the sense that no
infinite series coefficients appear. In order to simplify the
eventual calculation, we will work with a related function
\begin{equation}
\label{tilde f definition}
\tilde{f}=f-T_1-T_{s_0}=\sum_{m=1}^{\infty}q^{-2m}\left(
\underbrace{T_{s_0(s_1s_0)^m}}_{A}+\underbrace{T_{(s_1s_0)^m}}_{B}-
q\left(\underbrace{T_{(s_0s_1)^m}}_{C}+\underbrace{T_{s_1(s_0s_1)^m}}_{D}\right)
\right).
\end{equation}
The first thing is again to calculate the coefficients $b_w$ such
that $\tilde{f}=\sum_{w\in\tilde{W}}b_wC'_w$. For coefficients
$b_{s_0s_1}$, we see that instances of $C'_w$ are contributed by the $C$- and $D$-type terms starting from
$m=n$, and that, for length reasons, almost all the contributions cancel, leaving just $-qq^{-n}$. The
type $A$ terms contribute starting from $m=n$, and the type $B$ terms, from $m=n+1$. For the same reason,
only the first instance of $C'_{(s_0s_1)^n}$ coming from $T_{(s_0s_1)^n}$ fails to cancel, so that
$b_{(s_0s_1)^n}=q^n(-1-q)$.

No terms $C'_{(s_1s_0)^n}$ appear. Indeed, $A$- and $B$-type terms both begin contributing at $m=n$, but
have contributions with opposite signs. The same goes for $C$- and $D$-type terms, which both start contributing
from $m=n+1$. For exactly the same reasons (except the $A$ and $B$-type terms start to contribute at $m=n+1$
as well), no terms $C'_{s_1(s_0s_n)^n}$ appear.

For $b_{s_0(s_1s_0)^n}$, the $A$-type terms contribute from $m=n$ onwards, and the $B$-type terms, from $m=n+1$.
All contributions except the first cancel, leaving $q^{-n+\frac{1}{2}}$. The type $C$ and $D$ terms contribute
from $m=n+1$ and $=n+2$, respectively, with opposite signs as usual. Their contribution simplifies to
$qq^{-n-\frac{3}{2}}$, making $b_{s_0(s_1s_0)^n}=q^{-n}(q^{\frac{1}{2}}+q^{-\frac{1}{2}})$.

Therefore
\begin{equation}
\label{tilde f prepared form}
\tilde{f}=\sum_{n=1}^{\infty}q^{-n}(-1-q)C'_{(s_0s_1)^n}+q^{-n}\left(q^{\frac{1}{2}}+q^{-\frac{1}{2}}\right)C'_{s_0(s_1s_0)^n}.
\end{equation}

Recall from \S\ref{section the functions f and g} the functions $f$ and $g$ defined in \cite{BK} that form a basis of $J/H$.
\begin{theorem}
\label{images theorem}
We have
\begin{enumerate}
\item
$\phi(j(g))=\left(1+2\frac{q}{1-q}\right)t_1$;
\item
$\phi(j(\tilde{f}))=\left(q^{\frac{1}{2}}+q^{-\frac{1}{2}}\right)t_{s_0s_1}-(q+1)t_{s_0}$.
\end{enumerate}
\end{theorem}
\begin{proof}
Applying $j$ to equation \eqref{g expression C'}, we get $j(g)=
\left(1+2\frac{q}{1-q}\right)\sum_{w\in\tilde{W}}q^{\frac{\ell(w)}
{2}}C_w$. We conclude by adding the results of Lemma \ref{phi of si
lemma} together and recalling that $\phi$ preserves units.

Applying $j$ to expression \eqref{tilde f prepared form}, we obtain
\[
j(\tilde{f})=(1-q^{-1})\sum_{n=1}^{\infty}q^nC_{(s_0s_1)^n}+q^{n+\frac{1}{2}}C_{s_0(s_1s_0)^n},
\]
to which we apply Lemma \ref{phi inverse pair lemma}.
\end{proof}	
%

%
%
%
\section{The elements $t_w$ as functions on $G$}
\label{tw section}

\subsection{The Harish-Chandra Schwartz space}
\label{Schwartz space subsection}
From now on, we write $t_w$ for $(\phi\circ j)^{-1}(t_w)$ and we view $q$ as the cardinality of the residue field of $F$.

Recall that we can interpret $H$ as the convolution algebra
$C_c^\infty(I\rquotient G/I)$. Using Corollary \ref{tw Tw expansion
corollary}, we can see in an elementary way that the functions $t_y$
lie in the Harish-Chandra Schwartz space $\mathcal{C}(G)$, whose
definition we now recall.

Write $G=KAK$ where $K=\SL_2(\Oo)$
and $A$ is the maximal torus of diagonal matrices. We can write
any $g\in G$ as $g=k_1\pi^{\lambda(g)} k_2$, where $k_1,k_2\in K$
and $\lambda(g)$ is a dominant coweight depending on $g$
\textit{i.e.} in our case identifiable with a nonnegative integer.
Define $\Delta(g)=q^{\pair{\lambda}{\rho}}$, where $\rho$ is
the half-sum of positive roots. The \emph{Harish-Chandra Schwartz space}
is then the space of functions $f\colon G\to \C$ such that
$f$ is bi-invariant with respect to some open compact subgroup,
and such that for all polynomial functions $p\colon G\to F$
and $m>0$, we have
\begin{equation}
\label{Schwartz definition inequality}
\Delta(g)|f(g)|\leq \frac{C}{(\log(1+|p(g)|))^m}
\end{equation}
for some constant $C$ depending on $m$ and $p$.
\begin{prop}
The functions defined in Corollary \ref{tw Tw expansion corollary}
all lie in $\mathcal{C}(G)$.
\end{prop}
\begin{proof}
Clearly the $t_y$ are all bi-invariant with respect to the Iwahori
subgroup, which is open and closed in the compact subgroup $K$, as
it is the preimage of the discrete group $B(\F_q)$, hence is open
compact. Fix $y$ and let $f=t_y$.

Let $g\in K\pi^{\lambda}K= I\pi^\lambda I\sqcup Is_0\pi^{\lambda} I
\sqcup I\pi^\lambda s_0I\sqcup I\pi^{-\lambda} I$ for $\lambda=\lambda(g)=n>0$. Thus
$g$ lies in an Iwahori double coset corresponding to an element of $
\tilde{W}$ of length $2n\pm 1$. Here $\pi^\lambda$
is $(s_0s_1)^n$.
In our case, $\Delta(g)=q^{\lambda(g)}$, and so by Corollary
\ref{tw Tw expansion corollary}, up to a multiplicative
scalar depending on $f$ we have $\Delta(g)|f(g)|\leq q^{-n+ 2}$
if $\lambda$ is identified with $n$.
We must therefore bound $q^{2-n}(\log(1+|p(g)|))^{m}$ uniformly in $n$.
If $\lambda(g)=0$, then $\Delta(g)|f(g)|\leq q^2$ up to the same scalar.
Let $p$ and $m$ be given. Then
\[
p(g)=p(k_1ak_2)=\sum_{i=-N_1}^{N_2}(\pi^\lambda)^ip_i(k_1,k_2)
\] 
where the $p_i$ are polynomials in the eight entries of $k_1$ and
$k_2$, and $N_1,N_2\in \N$. Therefore
\[
|p(g)|\leq\max_i\left|(\pi^\lambda)^ip_i(k_1,k_2)\right|
\leq\max_i|\pi^{ni}|C_p\leq q^{nM_p}C_p
\]
for $C_p>0$ and $M_p\in \N$ depending on $p$. Then
\begin{align*}
\log(1+|p(g)|)&\leq\log(q^{nM_p}+q^{nM_p}C_p)
\\
&=
\log(q^{nM_p}(1+C_P))
\\
&
=nM_p\log(q(1+C_p)^{1/nM_p})
\\
&
\leq
nM_p\log(q(1+C_p))
\\
&
=nM_pD_p
\end{align*}
with $D_p>0$. Therefore $M_p^mD_p^m(\log (1+|p(g)|)^{-m}\geq n^{-m}$.
By elementary calculus, there is $F_m>0$ such that $n^m\leq F_m q^n$
for all $n\in \N$. It follows that
\[
\frac{1}{q^{n+2}}\leq\frac{1}{q^{n-1}}\leq \frac{q^2F_mM_p^mD_p^m}{(\log(1+|p(g)|))^m}
\]
as required.
\end{proof}
\subsection{Action on functions on the plane}
\subsubsection{The plane}
\label{the plane subsection}
Let $N=N(F)$ be the subgroup of upper triangular matrices with $1$s on
the diagonal, and recall that $G/N=F^2\setminus\{0\}$. Recalling
the Iwasawa decomposition $G=KAN$, where $K=\SL_2(\Oo)$ and $A$ is the maximal torus of diagonal matrices, we see that $K$-orbits in
$F^2\setminus\{0\}$ are labelled by $\Z=X_*(A)$, and are of the form
\[
K\pi^n\begin{pmatrix}
1 \\
0
\end{pmatrix}=\begin{pmatrix}
\pi^n e \\
\pi^n g
\end{pmatrix} .
\]
if elements of $K$ are written $k=\begin{pmatrix}
e & f \\
g & h
\end{pmatrix} $. Note that we cannot have both $e$ and $g$ divisible
by $\pi$, and therefore $K$-orbits are precisely of the form $\pi^n
\Oo^2\setminus\pi^{n+1}\Oo^2$. Indeed, $e$ and $g$ are 
not both in $\pi\Oo$, so one is a unit. If $e$ is a unit, then we may chose
$k=\begin{pmatrix}
e & 0 \\
g & e^{-1}
\end{pmatrix} $. If $g$ is a unit, we may chose
$k=\begin{pmatrix}
e & -g^{-1} \\
g & 0
\end{pmatrix} $.

Each $K$-orbit decomposes into two $I$-orbits. The two cases
that partition the points $k\pi^n(1,0)^T$  are $k
\in I$ and $k\not\in I$. If $k\in I$, then the $I$-orbit consists of
points of the form
\[
\begin{pmatrix}
\pi^n e \\
\pi^{n+1}g
\end{pmatrix}\in \begin{pmatrix}
\pi^n\Oo^\times \\
\pi^{n+1}\Oo
\end{pmatrix}\subset\pi^n\Oo^2\setminus\pi^{n+1}\Oo^2.
\]
We denote the characteristic functions of such orbits by $\psi_n$.
The remaining orbit consists of points of the form
\[
\begin{pmatrix}
\pi^ne \\
\pi^{n+1}g
\end{pmatrix}\in \begin{pmatrix}
\pi^n\Oo \\
\pi^{n}\Oo^\times
\end{pmatrix}\subset\pi^n\Oo^2\setminus\pi^{n+1}\Oo^2.
\]
We denote the characteristic functions of such orbits by $\varphi_n$.
The characteristic functions of the closures of these orbits are
\[
\bar{\varphi}_n:=\sum_{k=n}^{\infty}\varphi_k+\psi_k
\]
and
\[
\bar{\psi}_n:=\sum_{k=n}^{\infty}\psi_k+\varphi_{k+1}.
\]
The Iwahori subgroup acts on functions on $G/N$ by translation as $
(g\cdot f)(x)=f(g^{-1}x)$, and the functions $\bar{\varphi}_n$
and $\bar{\psi}_n$ give a basis for $C^\infty_c(F^2)^I$. Note that we have, for example,
$\varphi_0=\bar{\varphi}_0-\bar{\psi}_0$. The functions
$\bar{\varphi}_n$ give a basis for $C^\infty_c(F^2)^K$.

Recall also that $I\rquotient G/NA(\Oo)\simeq \tilde{W}$, hence $I$-invariant
functions (which are automatically $A(\Oo)$-invariant) on $F^2\setminus\{0\}$ are the same as functions on the
set of alcoves; in our case, intervals in $\R$ with integer endpoints.
A basis for $C_c^\infty(F^2)^I$ is then given under
this identification by half lines with integer boundary points,
corresponding to semi-infinite orbit closures.
For the general construction with a different normalization, see \cite{BKBasicAffine}.
We now
fix some relevant notation and identifications for alcoves.
We identify the alcove corresponding to $\varphi_0$ with the
interval $[-1,0]$ and the alcove corresponding $\psi_0$ with the
interval $[0,1]$, so that \textit{e.g.} $\varphi_2$ corresponds to
$[3,4]$. 
\subsubsection{Convolutions}
\label{subsection convolutions}
We can now describe how the affine Hecke algebra
acts on functions on the plane. The content of the following lemmas is
well known; for a general combinatorial description of them with
different normalizations, see \cite{Wgraph}. It will be useful to
observe that the convolution action commutes with the right action
of $2\Z$ on the set of alcoves, and that the functions $\varphi_n,\psi_n$ are periodic in the sense that $(m\alpha^\vee)\cdot\varphi_n=\varphi_{n+m}$ and likewise for $\psi_n$.

We view the convolution action as follows: given $T_w$ and the
characteristic function $\chi_X$ of an $I$-orbit $X$, we have a
multiplication map
\[
IwI\times X\to G/N,
\]
which descends to the quotient of the left-hand side by the equivalence
relation $(g,x)\sim (gi,i^{-1}x)$ for $i\in I$, yielding a map
\[
IwI\underset{I}{\times}X\to G/N.
\]
The image of this map is finitely-many $I$-orbits, and
the coefficient of the characteristic function of each orbit
is the number of points in the fibre over any point in that orbit.

It will be useful to note that $T_{s_0}$ and $T_{s_1}$ are related
by the following automorphism $\Phi$ of $G$. Let $\Theta$ be the automorphism
given by inverse-transpose, $\Psi$ be conjugation
by $\diag(1,\pi)\in\GL_2(F)$, and then $\Phi=\Psi\circ\Theta$.
Observe that $\Phi$ preserves $I$, and therefore induces an automorphism
of $H$, which exchanges $T_{s_0}$ and $T_{s_1}$. In particular,
$T_{s_1}$ can be realized as the characteristic function of $K'\setminus
I$, where $K'$ is the maximal compact subgroup
\[
\sets{\begin{pmatrix}
a & b \\
c & d
\end{pmatrix} }{a,d\in\Oo,~c\in\pi\Oo,~b\in\pi^{-1}\Oo}.
\]
The complement of $I$ is then the subset of such matrices with
$b\in \pi^{-1}\Oo^\times$.
\begin{lem}
\label{simple reflection action on functions on plane formulas}
We have
\begin{enumerate}
\item
$T_{s_0}\star \psi_n=\varphi_n$;
\item
$T_{s_0}\star \varphi_n=(q-1)\varphi_n+q\psi_n$;
\item
$T_{s_1}\star\varphi_n=\psi_{n-1}$;
\item
$T_{s_1}\star\psi_n=(q-1)\psi_n+q\varphi_{n+1}$.
\end{enumerate}
\end{lem}
\begin{proof}
By periodicity of $\varphi_n$ and $\psi_n$ and the fact that the action of $H$ commutes with translation, it suffices to prove the formulas in the case $n=0$.
To prove the first formula, let $g=\begin{pmatrix}
a & b \\
c & d
\end{pmatrix} \in K\setminus I:=Y$ \textit{i.e.} with $c\in\Oo^\times$ and let $\vb{x}$ be an
element in the orbit $X$ corresponding to $\psi_0$. Then $\vb{x}=(x,y)$
with $x\in\Oo^\times$ and $y\in \pi\Oo$, and
\[
gx=
\begin{pmatrix}
ax+by \\
cx+dy
\end{pmatrix}
\]
so that $cx+dy\in\Oo^\times$, and $ax+by$ is obviously integral.
Thus $T_{s_0}\star\psi_0$ is proportional to $\varphi_0$. To
prove the formula it remains to show that all fibres have size one.
Without loss of generality the situation is $g_1(1,0)=g_2(1,0)$
\textit{i.e.} the first columns of $g_1$ and $g_2$ agree. It follows
that $g_2^{-1}g_1\in N^+(\Oo)$, which stabilizes $(1,0)$ in $N^+\cap
I$. Therefore all fibres have size one.

To prove the second formula, let $g$ be as above and let $\vb{x}
=(x,y)\in\Oo^2$ with $y\in\Oo^\times$. Then $g\vb{x}$ is an integral
vector, and does not lie in $\pi\Oo^2$ as $\vb{x}$ is nonzero
modulo $\pi$, and $g$ is invertible modulo $\pi$. Therefore $T_{s_0}
\star\varphi_0$ is a linear combination of $\varphi_0$ and $\psi_0$.
Consider the map
\[
\xi\colon
\begin{pmatrix}
a & b \\
c & d
\end{pmatrix}
\mapsto\frac{a}{c}\mod\pi
\]
into $\F_q$,
which descends to the quotient $Y/I$. Therefore the fibre
over any point $(x,y)$ in either orbit injects into $\F_q$.
In the case where $y\in\Oo^\times$, then taking the fibre over $\vb{x}
=(0,-1)$ we see that $a\in\Oo^\times$, so that $\xi$ is into $\F_q^\times$ in this case. If $a\in\F_q^\times$, then
\[
\begin{pmatrix}
a & 0 \\
1 & a^{-1}
\end{pmatrix}
\begin{pmatrix}
0 \\
a
\end{pmatrix}
=
\begin{pmatrix}
0 \\
1
\end{pmatrix}
\in
\begin{pmatrix}
\Oo \\
\Oo^\times
\end{pmatrix}
\]
is a product of a matrix in $K\setminus I$ with a vector in the orbit
corresponding to $\varphi_0$. This shows that the coefficient of $\varphi_0$ is $q-1$. For any $a\in\F_q$, we have
\[
\begin{pmatrix}
a & -1 \\
1 & 0
\end{pmatrix}
\begin{pmatrix}
0 \\
-1
\end{pmatrix}
=
\begin{pmatrix}
1 \\
0
\end{pmatrix}
\in
\begin{pmatrix}
\Oo^\times \\
\pi\Oo
\end{pmatrix}.
\]
Therefore the coefficient of $\psi_0$ is $q$.

The case for the third formula is similar: if the matrices with entries
$a_i,b_i,c_i,d_i$ are in $I$, then
\begin{equation}
\label{Ts1 convolution product varphi0 equation}
\begin{pmatrix}
a_1 & b_1 \\
c_1 & d_1
\end{pmatrix}
\begin{pmatrix}
 & \pi^{-1} \\
-\pi &
\end{pmatrix}
\begin{pmatrix}
a_2 & b_2 \\
c_2 & d_2
\end{pmatrix}
\begin{pmatrix}
0  \\
1
\end{pmatrix}
=
\begin{pmatrix}
\pi^{-1}a_1d_1-\pi b_1b_2 \\
\pi^{-1}c_1d_2-\pi b_2d_1
\end{pmatrix}
\end{equation}
has top entry in $\pi^{-1}\Oo^\times$ and bottom entry in $\Oo$. Indeed,
$\pi\nmid a_1$ and $\pi\nmid d_2$, and $\pi\mid c_1$,
so the bottom row of \eqref{Ts1 convolution product varphi0 equation} is integral. Therefore $T_{s_1}\star\varphi_0$ is
proportional to $\psi_{-1}$. To show the fibres all have size one,
we can again calculate that any two matrices of the above form whose
right columns agree are in the same $N^{-}(\Oo)\cap I=\stab{I}{(0,1)}$
coset.

For the fourth formula, the fact that we have
\begin{equation}
\label{Ts1 convolution psi0 product equation}
\begin{pmatrix}
a_1 & b_1 \\
c_1 & d_1
\end{pmatrix}
\begin{pmatrix}
 & \pi^{-1} \\
-\pi &
\end{pmatrix}
\begin{pmatrix}
a_2 & b_2 \\
c_2 & d_2
\end{pmatrix}
\begin{pmatrix}
1  \\
0
\end{pmatrix}
=
\begin{pmatrix}
a_1c_2\pi^{-1}-a_2b_1\pi
 \\
c_1c_2\pi^{-1}-a_2d_1\pi
\end{pmatrix}
\in
\begin{pmatrix}
\Oo \\
\pi\Oo
\end{pmatrix}
\end{equation}
is clear.
We want to see that these products lie in
\[
\begin{pmatrix}
\Oo^\times \\
\pi\Oo
\end{pmatrix}
\coprod
\begin{pmatrix}
\pi\Oo \\
\pi\Oo^\times
\end{pmatrix}\subset\begin{pmatrix}
\Oo \\
\pi\Oo
\end{pmatrix}.
\]
The complement of the disjoint union in $(\Oo,\pi\Oo)^T$ is $(\pi\Oo,\pi^2\Oo)^T$. Any
matrix in $K'$ with its left column in the complement would have determinant in $\pi\Oo$,
and so the products all lie in the disjoint union. Therefore $T_{s_1}\star
\phi_0$ is a linear combination of $\psi_0$ and $\varphi_1$. To count
points in the fibre, we will use that $T_{s_1}=\chi_{K'\setminus I}$.
Define $\xi'\colon K'\setminus I\to \F_q$ by
\[
\xi'\colon
\begin{pmatrix}
a & b \\
c & d
\end{pmatrix}
\mapsto
\frac{d}{\pi b}\mod\pi,
\]
and note this function is right $I$-invariant. For any $d\in\F_q$, we
have that
\[
\begin{pmatrix}
0 & \pi^{-1} \\
-\pi & d
\end{pmatrix}
\begin{pmatrix}
-1 \\
0
\end{pmatrix}
=
\begin{pmatrix}
0   \\
\pi
\end{pmatrix}
\in
\begin{pmatrix}
\pi\Oo \\
\pi\Oo^\times
\end{pmatrix}
\]
is the product of a matrix in $K'\setminus I$ and a vector in $X$.
Therefore the coefficient of $\varphi_1$ is $q$. Taking the fibre
over $(1,0)$, we see that $d\in\Oo^\times$, so that $\xi'$ is into
$\F_q^\times$ in this case. If $d\in\F_q^\times$, then
\[
\begin{pmatrix}
d^{-1} & \pi^{-1} \\
0 & d
\end{pmatrix}
\begin{pmatrix}
 d \\
 0
\end{pmatrix}
=
\begin{pmatrix}
1 \\
0
\end{pmatrix}
\in
\begin{pmatrix}
\Oo^\times \\
\pi\Oo
\end{pmatrix}
\]
shows that the coefficient of $\psi_0$ is $q-1$.
\end{proof}
Assembling the formulas from Lemma \ref{simple reflection action on functions on plane formulas} and the definitions of $\bar{\varphi}_n$ and $\bar{\psi}_n$ recovers the following fact.
\begin{cor}
\label{H acts corollary}
The Iwahori-Hecke algebra $H$ acts on $C_c^\infty(F^2)$. We have
\begin{enumerate}
\item 
$T_{s_0}\star\bar{\varphi}_n=q\bar{\varphi}_n$;
\item
$T_{s_1}\star\bar{\psi}_n=q\bar{\psi}_n$;
\item
$T_{s_0}\star\bar{\psi}_n=\bar{\varphi}_n-\bar{\psi}_n+q\bar{\varphi}_{n+1}$;
\item
$T_{s_1}\star\bar{\varphi}_n=\bar{\psi}_{n-1}-\bar{\varphi}_{n}+q\bar{\psi}_{n}$.
\end{enumerate}
\end{cor}
\begin{lem}
\label{general Tw action on functions on plane formulas}
We have
\begin{enumerate}
\item
$T_{(s_1s_0)^n}\star\psi_m=\psi_{m-n}$;
\item
$T_{s_0(s_1s_0)^n}\star\psi_m=\varphi_{m-n}$;
\item
$T_{(s_0s_1)^n}\star\varphi_m=\varphi_{m-n}$;
\item
$T_{s_1(s_0s_1)^n}\star\varphi_m=\psi_{m-n-1}$;
\item
\[
T_{(s_1s_0)^n}\star\varphi_m=q^{2n}\varphi_{m+n}+(q-1)\sum_{k=1}^{2n}q^{2n-k}\psi_{m+n-k};
\]
\item
\[
T_{s_0(s_1s_0)^n}\star\varphi_m=q^{2n+1}\varphi_{m+n}+(q-1)
\sum_{k=0}^{2n}q^{2n-k}\varphi_{m+n-k};
\]
\item
\[
T_{(s_0s_1)^n}\star\psi_m=q^{2n}\psi_{m+n}+(q-1)\sum_{k=1}^{2n}
q^{2n-k}\varphi_{m+n+1-k};
\]
\item
\[
T_{s_1(s_0s_1)^n}\star\psi_m=q^{2n+1}\varphi_{m+n+1}+
(q-1)\sum_{k=0}^{2n}q^{2n-k}\psi_{m+n-k}.
\]
\end{enumerate}
\end{lem}
\begin{proof}
Formulas 1--4 follow directly from Lemma \ref{simple reflection
action on functions on plane formulas}, and the remaining formulas
follow from 1--4 and another application of the lemma. For example, to 
prove formula 1, write $T_{(s_1s_0)^n}=T_{s_1}T_{s_0}\cdots T_{s_1}T_{s_0}$ and successively apply formulas 1 and 3 from Lemma \ref{simple reflection action on functions on plane formulas}. Formula 
5 is proved by induction on $n$, the base case being 
\[
T_{s_1s_0}\star\varphi_m=T_{s_1}T_{s_0}\star\varphi_m=q^2\varphi_{m+1}+(q-1)(q\psi_m+\psi_{m-1}),
\]
which again follows from Lemma \ref{simple reflection action on 
functions on plane formulas}, formulas 2, 3, and 4. Then by induction we have
\begin{align*}
T_{s_1s_0}T_{(s_1s_0)^n}\star\varphi_m&=T_{s_1s_0}\star q^{2n}\varphi_{m+n}+(q-1)\sum_{k=1}^{2n}q^{2n-k}\psi_{m+n-k}\\
&=q^{2n+2}\varphi_{m+n+1}+(q-1)q^{2n}\left(q\psi_{m+n}+\psi_{m+n-1}\right)+(q-1)\sum_{k=1}^{2n}q^{2n-k}\psi_{m+n-1-k}\\
&=q^{2n+2}\varphi_{m+n+1}+(q-1)\left(q^{2n+1}\psi_{m+n}+q^{2n}\psi_{m+n-1}+\sum_{k=3}^{2n+2}q^{2n+2-k}\psi_{m+n+1-k}\right),
\end{align*}
where between the first and second line we used the base case and 
formula 1 of this lemma.
\end{proof}
\begin{rem}
\label{general Tw action on functions on plane formulas
specification remark}
Observe that the formulas in Lemma \ref{general Tw action on
functions on plane formulas} recover those of Lemma \ref{simple
reflection action on functions on plane formulas} upon specifying
$n$, provided that sums with decreasing indices are interpreted
as empty.
\end{rem}
We can now describe the action of $J$ on functions on the plane. To
begin with, we present an elementary proof of the result from
the discussion following equation 4.1 in \cite{BK}, namely that
$t_1$ acts trivially.
\begin{prop}
\label{t1 plane action trivial prop}
We have $t_1\star \psi_m=t_1\star\varphi_m=0$ for all $m$.
\end{prop}
\begin{proof}
It suffices to check that $g$  (identified with a scalar multiple
of $t_1$ by theorem \ref{images theorem}) acts trivially, and for this it suffices to check that
$g\star\varphi_0=g\star\psi_0=0$. Now, $g$ sends $\psi_0$ to
\begin{align*}
\psi_0
-q^{-1}(q-1)(\varphi_0+q\varphi_1+(q-1)\psi_0)
&+q^{-2}\left(q^2\psi_1+(q-1)\left(q\varphi_1+\varphi_0\right)
+\psi_{-1}
\right)
\\
&
-q^{-3}
\left(
\varphi_{-1}+q^3\varphi_2+(q-1)\left(q^2\psi_{1}+q\psi_0+\psi_{-1}
\right)
\right)
\\
&+
q^{-4}\left(
\psi_{-2}+q^4\psi_2+(q-1)\left(q^3\varphi_2+q^2\varphi_1+q\varphi_0+
\varphi_{-1}\right)\right)
\\
&-
q^{-5}
\left(\varphi_{-2}+q^5\varphi_3+(q-1)\left(
q^4\psi_2+q^3\psi_1+q^2\psi_0+q\psi_{-1}+\psi_{-2}
\right)\right)
\\
&
+\cdots
\end{align*}
and after cancellations between these terms we are left with
\[
-q^4\left(q^3\varphi_2+q^2\varphi_1+q\varphi_0+
\varphi_{-1}\right)
-
q^{-5}
\left(\varphi_{-2}+q^5\varphi_3-\left(
q^4\psi_2+q^3\psi_1+q^2\psi_0+q\psi_{-1}+\psi_{-2}
\right)\right)+\cdots
\]
Further, all cancellation of terms
corresponding to elements of length $l$ occurs between terms
corresponding to lengths $l\pm 2$, and proceeds as follows. We have
\begin{align}
&-q^{-2n+1}
\left(
\varphi_{-n+1}+q^{2n-1}\varphi_n+(q-1)\sum_{k=0}^{2n-2}q^{2n-2-k}
\psi_{n-1-k}
\right)
\label{s0(s1s0)n-1 and s1(s0s1)n-1}
\\
&+q^{-2n}
\left(
\psi_{-n}+
q^{2n}\psi_n+(q+1)\sum_{k=1}^{2n}q^{2n-k}\varphi_{n+1-k}
\right)
\label{(s1s0)n and (s0s1)n}
\\
&-q^{-2n-1}
\left(
\varphi_{-n}+q^{2n+1}\varphi_{n+1}+(q-1)\sum_{k=0}^{2n}q^{2n-k}
\psi_{n-k}
\right)
\label{s0(s1s0)n and s1(s0s1)n}
\\
&+
q^{-2n-2}
\left(
\psi_{-n-1}+q^{2n+2}\psi_{n+1}+(q-1)\sum_{k=1}^{2n+2}q^{2n+2-k}
\varphi_{n+2-k}
\right)
\label{(s1s0)n+1 and (s0s1)n+1}
\\
&-q^{-2n-3}
\left(
\varphi_{-n-1}+q^{2n+3}\varphi_{n+2}+(q-1)\sum_{k=0}^{2n+2}q^{2n+2-k}\psi_{n+1-k}
\right),
\label{s0(s1s0)n+1 and s1(s0s1)n+1}
\end{align}
where line \eqref{s0(s1s0)n-1 and s1(s0s1)n-1} corresponds to
$T_{s_0(s_1s_0)^{n-1}}\star\psi_0+T_{s_1(s_0s_1)^{n-1}}\star\psi_0$,
line \eqref{(s1s0)n and (s0s1)n} corresponds to
$T_{(s_1s_0)^n}\star\psi_0+T_{(s_0s_1)^n}\star\psi_0$ and so on up
to line \eqref{s0(s1s0)n+1 and s1(s0s1)n+1} corresponding to
$T_{s_0(s_1s_0)^{n+1}}\star\psi_0+T_{s_1(s_0s_1)^{n+1}}\star\psi_0$.

We will explain the cancellation of the terms in line
\eqref{s0(s1s0)n and s1(s0s1)n};
the cancellation of terms in odd-numbered lines follows the same pattern. The lead term in line \eqref{s0(s1s0)n and s1(s0s1)n} cancels with the final term in $q$
times the sum in line \eqref{(s1s0)n+1 and (s0s1)n+1}, and the
second cancels with the first term in $q$ times the sum. The first
and last terms in $q$ times the sum in line \eqref{s0(s1s0)n and
s1(s0s1)n} cancel with the leading terms of line \eqref{(s1s0)n and
(s0s1)n}, and the middle terms cancel
with $-1$ times the sum in line \eqref{s0(s1s0)n-1 and s1(s0s1)n-1}.
The terms in $-1$ times the sum in line \eqref{s0(s1s0)n and
s1(s0s1)n} cancel with the middle terms of $q$ times the sum in line
\eqref{s0(s1s0)n+1 and s1(s0s1)n+1}.

The cancellations in $g\star\varphi_0$ follow the same pattern.
\end{proof}
\begin{lem}
\label{starting with series plane action lemma}
We have (note that none of the sums below contains a $T_1$ term)
\begin{enumerate}
\item
\label{starting with so on varphi0 formula}
\[
\sum_{\substack{w\in\tilde{W} \\ w~\text{\emph{starts with}}~s_0}}
(-1)^{\ell(w)}q^{-\ell(w)}T_w\star\varphi_m
=-\bar{\varphi}_{m};
\]
\item
\label{starting with s1 on varphi0 formula}
\[
\sum_{\substack{w\in\tilde{W} \\ w~\text{\emph{starts with}}~s_1}}
(-1)^{\ell(w)}q^{-\ell(w)}T_w\star\varphi_m
=\bar{\psi}_m
;
\]
\item
\[
\sum_{\substack{w\in\tilde{W} \\ w~\text{\emph{starts with}}~s_0}}
(-1)^{\ell(w)}q^{-\ell(w)}T_w\star\psi_m
=\bar{\varphi}_{m+1};
\]
\item
\[
\sum_{\substack{w\in\tilde{W} \\ w~\text{\emph{starts with}}~s_1}}
(-1)^{\ell(w)}q^{-\ell(w)}T_w\star\psi_n
=-\bar{\psi}_n.
\]
\end{enumerate}
\end{lem}
\begin{proof}
It suffices by periodicity of $\varphi_m,\psi_m$ to prove the lemma for $m=0$.
We evaluate each convolution term-by-term, and then explain the
cancellations that occur between adjacent terms. After accounting
for the contributions of the first few terms, this gives the results
of the lemma.

In the case of formula \ref{starting
with so on varphi0 formula}, we have adjacent terms of the form
\begin{multline*}
-q^{-2n+1}\left(
\underbrace{
\overbrace{q^{2n+1}\psi_{n-1}}^{D}+(q-\overbrace{1}^{A})\sum_{k=0}^{2n-2}q^{2n-2-k}\varphi_{n-1-k}
}_{T_{s_0(s_1s_0)^{n-1}}}
+
\right)
+q^{-2n}
\underbrace{\overbrace{\varphi_{n-1}}^{C}}_{T_{(s_0s_1)^n}}
\\
-q^{-2n-1}
\left(
q^{2n+1}\psi_n+(\overbrace{q}^{B}-1)\sum_{k=0}^{2n}q^{2n-k}\varphi_{n-k}
\right).
\end{multline*}
Adding the contributions $A+B+C+D$ gives $-(\varphi_n+\psi_n)$.
The other terms cancel out similarly by induction. Starring this
procedure from $n=1$ captures the contributions of all terms
starting from $T_{s_0}$, although we must add the contribution of the
first $D$- and $B$-type terms. Thus formula \ref{starting with so on
varphi0 formula} is proved.

In the case of formula \ref{starting with s1 on varphi0 formula},
we have adjacent terms of the form
\begin{multline*}
q^{-2n+2}\left(
\underbrace{
\overbrace{q^{2n-2}\varphi_{n-1}}^{H}+(q-\overbrace{1}^{E})\sum_{k=1}^{2n-2}q^{2n-2-k}\psi_{n-1-k}
}_{T_{(s_1s_0)^{n-1}}}
\right)
-q^{-2n+1}
\underbrace{\overbrace{\psi_{-n}}_{T_{s_1(s_0s_1)^{n-1}}}^{L}}
\\
+q^{-2n}
\left(
+q^{2n}\varphi_n+(\overbrace{q}^{F}-1)\sum_{k=1}^{2n}q^{2n-k}\psi_{n-k}
\right).
\end{multline*}
Adding terms $E+F+L+H$ gives $\varphi_{n-1}+\psi_{n-1}$.
We can start this cancellation from $n=2$, adding the contributions of
the first type $L$ and $F$ terms. This proves formula \ref{starting with
s1 on varphi0 formula}.

The remaining formulas follow the same pattern.
\end{proof}
\begin{prop}
\label{ts0 plane action formula prop}
For all $m$:
\begin{enumerate}
\item
We have $t_{s_0}\star\bar{\varphi}_m=\bar{\varphi}_m$, and
$t_{s_0}\star\bar{\psi}_m=0$. Thus $t_{s_0}$ acts by a projector
\[
C_c^\infty(F^2)^I\onto C_c^\infty(F^2)^K.
\]
\item
We have: $t_{s_1}\star\bar{\psi}_m=\psi_m$, and
$t_{s_1}\star\bar{\varphi}_m=0$. Therefore $t_{s_1}$ acts as $\id-t_{s_0}$.
\end{enumerate}
\end{prop}
\begin{proof}
It is enough to prove the proposition for $m=0$.
We first calculate $t_{s_0}\star (\varphi_0+\psi_o)$,
then using periodicity we will obtain formulas for
$t_{s_0}\star(\varphi_n+\psi_n)$. The last step will be to take
\[
t_{s_0}\star\bar{\varphi}_0=\sum_{n=0}^{\infty}t_{s_0}\star(\varphi_n+
\psi_n).
\]
Indeed, it follows from Corollary \ref{tw Tw expansion corollary}
and Lemma \ref{starting with series plane action lemma} that
\[
-q^{-1}(1+q)(t_{s_0}\star(\varphi_0+\psi_0))=-(1+q^{-1})(\varphi_0+
\psi_0)
\]
so that $t_{s_0}\star\bar{\varphi}_0=\bar{\varphi}_0$. The first
statement follows. Again
using periodicity to calculate $t_{s_0}\star(\psi_n+\varphi_{n+1})$, we get
that $t_{s_0}\star\bar{\psi}_n=0$. Therefore $t_{s_0}$ kills all
basis functions that are not $K$-invariant.

The calculation for $t_{s_1}$ is similar.
\end{proof}
\begin{rem}
\label{J ring structure remark}
It is in fact
easy to see using the ring structure on $J$ that $t_{s_0}$ and $t_{s_1}$ are idempotent.
\end{rem}
\begin{theorem}
\label{J action theorem}
The algebra $J$ acts on $C_c^\infty(F^2)^I$.
\end{theorem}
\begin{proof}
The last sentence of Proposition \ref{ts0 plane action formula prop}
says that the identity in $J$ acts on $C^\infty_c(F^2)^I$ by the identity endomorphism;
recall we have shown $t_1$ acts trivially in Proposition \ref{t1 plane
action trivial prop}. By Corollary \ref{H acts corollary}, the action of $H$ on $C_c^\infty(F^2)^I$ is well-defined. By Proposition \ref{ts0 plane action formula prop}, $t_{s_0}$ and $t_{s_1}$ have well-defined actions. Now using the first formula of lemma \ref{phi lemma}, we see that $t_{s_is_j}$ has a well-defined action. Then using the second formula of that lemma we see that $t_{s_is_js_i}$ has a well-defined action, and so on.
\end{proof}

\bibliography{sl_2_paper_biblio.bib}

\end{document}

%% file: preamble.tex
\newcommand{\vb}[1]{\mathbf{#1}}
\newcommand{\diag}{\mathrm{diag}}
\newcommand{\pair}[2]{\left\langle #1 , #2\right\rangle}

\DeclareMathOperator{\spn}{span}

\def\into{\mathrel{\hookrightarrow}}
\def\onto{\mathrel{\twoheadrightarrow}}

\newcommand{\sets}[2]{\left\{#1\,\middle|\,#2\right\}}
\newcommand{\genrel}[2]{\left\langle #1\,\middle|\,#2\right\rangle}

\newcommand{\stab}[2]{\mathrm{Stab}_{#1}(#2)}

\newcommand{\rquotient}{\backslash}

\DeclareMathOperator{\End}{End}

\newcommand{\id}{\mathrm{id}}

\newcommand{\R}{\mathbb{R}}
\newcommand{\N}{\mathbb{N}}
\newcommand{\Z}{\mathbb{Z}}
\newcommand{\C}{\mathbb{C}}

\newcommand{\F}{\mathbb{F}}
\newcommand{\Oo}{\mathcal{O}}

\newcommand{\GL}{\mathrm{GL}}
\newcommand{\SL}{\mathrm{SL}}

\newtheorem{theorem}{Theorem}
\newtheorem{cor}{Corollary}
\newtheorem{prop}{Proposition}
\newtheorem{lem}{Lemma}

\theoremstyle{definition}
\newtheorem{dfn}{Definition}

\theoremstyle{remark}
\newtheorem{ex}{Example}

%% file: sl_2_paper_v8.bbl
\begin{bibdiv}
\begin{biblist}

\bib{BK}{article}{
      author={Braverman, A.},
      author={Kazhdan, D.},
       title={Remarks on the asympotitic {H}ecke algebra},
        date={2018},
      eprint={arXiv:1704.03019},
         url={https://arxiv.org/abs/1704.03019},
}

\bib{BKBasicAffine}{article}{
      author={Braverman, A.},
      author={Kazhdan, D.},
       title={On the schwartz space of the basic affine space},
        date={1999Apr},
     journal={Selecta Math. (N.S.)},
      volume={5},
      number={1},
       pages={1\ndash 28},
}

\bib{KL79}{article}{
      author={Kazhdan, D.},
      author={Lusztig, G.},
       title={Representations of coxeter groups and hecke algebras},
        date={1979},
     journal={Invent. Math.},
      volume={53},
       pages={165\ndash 184},
}

\bib{LusUnequal}{article}{
      author={Lusztig, G.},
       title={{H}ecke algebras with unequal parameters},
        date={2014},
      eprint={arXiv:math/0208154},
         url={https://arxiv.org/abs/math/0208154},
}

\bib{affineII}{article}{
      author={Lusztig, G.},
       title={Cells in affine {W}eyl groups, {II}},
        date={1987},
     journal={J. Algebra},
      volume={109},
       pages={536\ndash 548},
}

\bib{Wgraph}{article}{
      author={Lusztig, G.},
       title={Periodic {$W$}-graphs},
        date={1997},
     journal={Represent. Theory},
      volume={1},
       pages={207\ndash 279},
}

\end{biblist}
\end{bibdiv}
